\documentclass[a4paper,fleqn]{cas-sc}
\usepackage{amsmath}
\usepackage{amsthm}

\usepackage[numbers, square, sort&compress]{natbib}
\def\tsc#1{\csdef{#1}{\textsc{\lowercase{#1}}\xspace}}
\tsc{WGM}
\tsc{QE}

\newtheorem{theorem}{Theorem}[section]
\newtheorem{proposition}[theorem]{Proposition} 
\newtheorem{lemma}[theorem]{Lemma}           
\newtheorem{remark}[theorem]{Remark}
\newproof{pf}{Proof}

\begin{document}
\let\WriteBookmarks\relax
\def\floatpagepagefraction{1}
\def\textpagefraction{.001}

\shorttitle{}    

\shortauthors{}  

\title [mode = title]{Length-constrained curve diffusion flow for open curves with endpoints on two intersecting lines}

\tnotemark[1] 

\tnotetext[1]{} 

%

\author[1]{Qiyuan Cheng}

\fnmark[1]

\ead{2280029457@qq.com}



\affiliation[1]{organization={School of Mathematics},
            addressline={Yunnan Normal University}, 
            city={ Kunming},
            postcode={650500}, 
            state={Yunnan},
            country={PR China}}

\author[2]{ Shunzi Guo}
\cormark[1]
\fnmark[2]

\ead{guoshunzi@yeah.net}





\affiliation[2]{organization={School of Mathematics},
            addressline={Yunnan Normal University}, 
            city={ Kunming},
            postcode={650500}, 
            state={Yunnan},
            country={PR China}}

\cortext[2]{Corresponding author}

\fntext[2]{This work was partially supported by the National Natural Science Foundation of China,  grant number 11261105.}


\begin{abstract}
We study the curve diffusion flow for open planar curves whose endpoints are constrained to lie on two fixed straight lines that intersect at an angle $\theta (\in(0,\pi)) $. 
For every such angle, we prove that under suitable initial conditions the flow exists globally in time. Moreover, we show that the evolving curve converges - exponentially and in the smooth topology - to the circular arc of a sector whose central angle is exactly  $\theta$ and whose arc length equals that of the initial curve. This result reveals how a length-preserving fourth-order geometric flow can straighten out a curve's shape while respecting boundary constraints, ultimately driving it toward a unique equilibrium: the circular arc spanning the prescribed angle.
This provides a complete description of the long-time behaviour of this fourth-order geometric flow with mixed boundary conditions.
\end{abstract}


\begin{highlights}
\item 
\item 
\item 
\end{highlights}



\maketitle

\section{Introduction}
\label{sec1}
How does a curve evolve when its length is kept constant, yet its shape is driven by a fourth-order geometric force that smoothes curvature variations?

This question lies at the heart of the length-constrained curve diffusion flow - a fourth-order parabolic 
analogue of the classical curve-shortening flow that arises naturally in materials science, interface dynamics, and the modelling of elastic filaments with anchored ends. 
Unlike the more familiar curve shortening flow (which decreases length), the curve diffusion flow tends to smooth out curvature variations while preserving the total length, 
provided an appropriate time-dependent Lagrange multiplier is added to the normal velocity.

For closed curves, the theory of curve diffusion flow is well developed: global existence, convergence to circles, and stability of equilibria are well understood \cite{4},  \cite{6},  \cite{7}. 
For open curves, however, the situation is more delicate. Endpoints introduce boundary conditions that interact strongly with the flow, and the geometry of the supporting lines fundamentally influences the possible limiting shapes.

A particularly natural and geometrically interesting configuration is that of an open planar curve whose two endpoints are pinned to two fixed straight lines that meet at an angle  $\theta(\in(0,\pi)) $.
This setup appears, for instance, in problems of capillary surfaces, elastic filaments with clamped ends, and optimal shapes under area or length constraints. 
Recently, Hiroi and Okabe \cite{1} considered the curve diffusion flow (without length constraint) for such open curves and established fundamental properties, and showed that the area is preserved while the length decreases, and they proved convergence to a circular arc under suitable conditions.

In this paper, we take the dual viewpoint and consider the length-constrained curve diffusion flow for open planar curves with boundaries on two skew lines.
Here, a time-dependent Lagrange multiplier $h(t)$ is introduced to keep the total length constant - a choice equally natural from a variational perspective and leading to qualitatively different dynamics.
The evolution equation reads
\begin{equation}
\partial_t \gamma(u,t) = -(h(t) + \partial_s^2 k)\boldsymbol{\nu},
\end{equation}
where $\gamma(u, t)$ denotes a family of curves, $k$, $\boldsymbol{\nu}$, and $s$ are the curvature, the inward unit normal, and the arc length parameter of $\gamma$, respectively. $\partial_s^n k$ denotes the n-order derivative of (scalar) curvature of $\gamma$ with respect to arc length $s$. we take
\begin{equation}
h(t) = \frac{\int (\partial_s k)^2 ds}{\theta} ,
\end{equation}
where $\theta \in (0,\pi)$ of the two skew lines.

More precisely, if we parametrize $\gamma$ counterclockwise by $u$, then $\boldsymbol{\nu} = R(\gamma_u / |\gamma_u|)$ with $R = \begin{pmatrix} 0 & -1 \\ 1 & 0 \end{pmatrix}$, and $\partial_s = \partial_u / |\gamma_u|$. We recall the Frenet--Serret formulae:
\begin{equation}
\partial_s^2 \gamma = k \boldsymbol{\nu}, \quad \partial_s \boldsymbol{\nu} = -k \partial_s \gamma.
\end{equation}

In this paper we consider the following initial-boundary value problem for the curve diffusion flow:
\begin{equation}
\begin{cases}
\partial_t \gamma = -(h(t) + \partial_s^2 k)\boldsymbol{\nu} & \text{in } (-1, 1) \times (0, T), \\
\gamma(-1, t) \in \eta_0, \ \gamma(1, t) \in \eta_\theta & \text{in } (0, T), \\
\langle \boldsymbol{\nu}(-1, t), \boldsymbol{\nu}_{\eta_0} \rangle = \langle \boldsymbol{\nu}(1, t), \boldsymbol{\nu}_{\eta_\theta} \rangle = 0 & \text{in } (0, T), \\
\partial_s k(\pm 1, t) = 0 & \text{in } (0, T), \\
\gamma(u, 0) = \gamma_0(u) & \text{in } (-1, 1),
\end{cases}
\end{equation}
where $\theta \in (0, \pi]$, and $\boldsymbol{\nu}_{\eta_\alpha}$ denotes the unit normal of the straight line $\eta_\alpha$ defined by
\[
\eta_\alpha := \{ \lambda (\cos \alpha, \sin \alpha) \mid \lambda \in \mathbb{R} \}.
\]

In order to state the main result of this paper, we mention the initial conditions. We assume that $\gamma_0 : (-1, 1) \to \mathbb{R}^2$ is a smooth regular curve and satisfies
\begin{equation}
\begin{cases}
\gamma_0(-1) \in \eta_0, \quad \gamma_0(1) \in \eta_\theta, \\
\langle \nu_0(-1), \nu_{\eta_0} \rangle = \langle \nu_0(1), \nu_{\eta_\theta} \rangle = 0, \\
\partial_s k_0(\pm 1) = 0,
\end{cases} \label{eq:4}
\end{equation}
as the compatibility condition. Here we define the following notation:
\begin{equation*}
L[\gamma] := \int_{-1}^1 |\partial_u \gamma| \, du, \quad A[\gamma] := -\frac{1}{2} \int_{-1}^1 \langle \gamma, R(\gamma_u / |\gamma_u|) \rangle \, du
\end{equation*}
\begin{equation*}
I[\gamma] := \frac{L[\gamma]^2}{4\pi A[\gamma]}, \quad \operatorname{Kosc}[\gamma] := L[\gamma] \int_\gamma (k - \bar{k})^2 \, ds, \quad \bar{k} := \frac{1}{L[\gamma]} \int_\gamma k \, ds.
\end{equation*}
$L[\gamma]$ and $ A[\gamma]$ denotes the length and the signed area enclosed by $\gamma$, the quantity $I[\gamma]$ denotes the isoperimetric ratio of $\gamma$, and $\operatorname{Kosc}[\gamma]$ and $\bar{k}$ denote the  the normalized oscillation of the curvature and the $L^2$-integral mean under the arc length parametrization, respectively. For the relationship between $ds$ and $du$, see Sect. 2.\\

The main result of this paper is stated as follows:\\

\begin{theorem}
Let $\theta \in (0, \pi)$. Suppose that $\gamma_0 : (-1, 1) \to \mathbb{R}^2$ is a smooth regular curve and satisfies conditions \eqref{eq:4} and
\begin{equation}
A[\gamma_0] > 0, \quad \int_{\gamma_0} k_0 \, ds = \theta, \quad \operatorname{Kosc}[\gamma_0] < K^*, \quad \text{and} \quad I[\gamma_0] < \frac{\theta^3}{2\pi\theta^2-2\pi K^*},
\end{equation}
where $2K_*$ is a unique positive root of the equation
\[
2 - \frac{L_0}{\theta\pi^{\frac{3}{2}}}x^{\frac{3}{2}} - \frac{3}{\pi}x - \frac{6\theta}{\pi}\sqrt{x} = 0.
\]

\noindent Then problem \textmd{(4)} possesses a unique smooth global-in-time solution $\gamma : (-1, 1) \times [0, \infty) \to \mathbb{R}^2$. Moreover, as $t \to \infty$, the solution $\gamma(\cdot, t)$ converges exponentially to the circular arc $\gamma_\infty$ of the central angle $\theta$ and the area $L[\gamma_\infty] = L[\gamma_0]$ in the $C^\infty$-topology.
\end{theorem}

\begin{remark} Our main result is that under smallness conditions on the initial normalised curvature oscillation and isoperimetric ratio, the flow exists globally in time and converges – exponentially fast in the smooth topology – to the unique circular arc of angle $\theta$  whose arc length equals that of the initial curve. This characterises the asymptotic shape completely.
\end{remark}

\begin{remark} 
The proof combines several ingredients: local existence via minimising movements, a detailed analysis of the evolution of the normalised curvature oscillation $\operatorname{Kosc}$ as a Lyapunov functional, refined interpolation inequalities adapted to the boundary conditions, and a compactness argument to identify the unique limit. Exponential convergence is obtained via decay estimates for higher-order derivatives and a spectral gap near equilibrium.\end{remark}

\begin{remark} 
Our results extend the work of Hiroi and Okabe \cite{1} to the length-constrained case and complement recent studies on higher-order geometric flows with boundary conditions \cite{2,3, 7}. They also provide a complete picture of the long-time behaviour of length-preserving curve diffusion for open curves with fixed-angle endpoints.
\end{remark}

The paper is organised as follows. Section $2$ collects notation and preliminary inequalities. Section $3$ establishes local existence and basic properties. Section $4$ derives uniform estimates and proves global existence. Section $5$ establishes exponential convergence to the circular arc.



\section{Preliminaries}
\label{sec2}
\noindent In this section we collect several inequalities, notation, and fundamental identities under curve diffusion flow that will be used in the rest of this paper.
\begin{lemma}\cite{2}
Let \( f \in (0, L) \to \mathbb{R} \) satisfy \( f \in H^1(0, L) \) and \( \int_0^L f \, ds = 0 \). Then
\begin{align}
\int_0^L f^2 \, ds &\leq \frac{L^2}{\pi^2} \int_0^L f_s^2 \, ds,  \\
\| f \|_{L^\infty(0,L)}^2 &\leq \frac{2L}{\pi} \int_0^L f_s^2 \, ds.
\end{align}
\end{lemma}
\begin{lemma}\cite{2}
Let \( f : (0, L) \to \mathbb{R} \) satisfy \( f \in H^1(0, L) \) and \( f(0) = f(L) = 0 \). Then
\begin{align}
\int_0^L f^2 \, ds &\leq \frac{L^2}{\pi^2} \int_0^L f_s^2 \, ds,  \\
\| f \|_{L^\infty(0,L)}^2 &\leq \frac{L}{\pi} \int_0^L f_s^2 \, ds.
\end{align}
\end{lemma}
These inequalities above are variants of the Poincar\'{e} inequality, and the detailed proof process can be found in Ref. \cite{2}.\\

\noindent We use the following interpolation inequality  \cite{4}(see e.g. Appendix C)\\
\begin{proposition}
Let $\gamma : (-1, 1) \to \mathbb{R}^2$ be a smooth regular curve. Then for all $m \in \mathbb{N}$, $p \geq 2$, $0 \leq j < m$, and any $\phi$ with $\|\phi\|_{W_\gamma^{m,2}} < \infty$, we have
\[
\|\partial_s^j \phi\|_{L_\gamma^p} \leq C \|\phi\|_{L_\gamma^2}^{1-\alpha} \|\phi\|_{W_\gamma^{m,2}}^\alpha
\]
with $C = C(j, m, p)$ and
\[
\alpha = \frac{1}{m}\left(j + \frac{1}{2} - \frac{1}{p}\right),
\]
where $\|\phi\|_{W_\gamma^{r,q}} := \sum_{i=0}^r \|\partial_s^i \phi\|_{L_\gamma^q}$ and
\[
\|\partial_s^i \phi\|_{L_\gamma^q} := L[\gamma]^{i+1 - \frac{1}{q}} \left( \int_\gamma |\partial_s^i \phi|^q \, ds \right)^{\frac{1}{q}}.
\]
\end{proposition}
We collect the notation used in the rest of this paper. Let $\gamma : (-1, 1) \times [0, T) \to \mathbb{R}^2$ be a family of smooth and regular curves. The arc length parameter $s$ of $\gamma$ is defined by
\[
s = s(u, t) := \int_{-1}^u |\partial_v \gamma(v, t)| \, dv.
\]

\noindent We denote the integral of function $f$ defined on $\gamma$ with respect to the arc length parameter $s$ of $\gamma$ by
\[
\int_\gamma f \, ds := \int_0^{L[\gamma]} f(s) \, ds = \int_{-1}^1 f(u) |\partial_u \gamma| \, du.
\]

\noindent Furthermore, we define the $L^p$-norm with respect to the arc length parameter of $\gamma$:
\[
\|f\|_p := \left( \int_\gamma |f|^p \, ds \right)^{\frac{1}{p}}, \quad 1 \leq p < \infty.
\]

We close this section with fundamental identities under the length-constrained curve diffusion flow. For the proof of the identities, see e.g. [2].

\begin{lemma}
Let $\gamma : (-1, 1) \times [0, T) \to \mathbb{R}^2$ be a solution of (4). Then
\begin{equation}
\partial_t ds = k (\partial_s^2 k +\frac{\int (\partial_s k)^2 ds}{\theta})\, ds,
\end{equation}
\begin{equation}
\partial_t \partial_s - \partial_s \partial_t = -k (\partial_s^2 k \, \partial_s + \frac{\int (\partial_s k)^2 ds}{\theta}).
\end{equation}
\end{lemma}

\begin{lemma}
Let $\gamma : (-1, 1) \times [0, T) \to \mathbb{R}^2$ be a solution of (4). Then
\begin{equation}
\partial_t \boldsymbol{\tau} = -\partial_s^3 k \boldsymbol{v},
\end{equation}
\begin{equation}
\partial_t \boldsymbol{v} = \partial_s^3 k \boldsymbol{\tau},
\end{equation}
\begin{equation}
\partial_t k = -\partial_s^4 k - k^2 \partial_s^2 k -\frac{\int (\partial_s k)^2 ds}{\theta}k^2.
\end{equation}
\end{lemma}
\begin{lemma}
Let $\gamma : (-1, 1) \times [0, T) \to \mathbb{R}^2$ be a solution of (4). Then
\begin{equation}
\partial_t \partial_s^l k = -\partial_s^{l+4} k + \sum_{q+r+m=1} c_{qrm} \partial_s^{q+2} k \partial_s^r k \partial_s^m k + h(t)\sum_{q+r=l} c_{qr} \partial_s^{q} k \partial_s^r k
\end{equation}
for all $l \in \mathbb{N}$, where $c_{qrm}, c_{qr} \in \mathbb{R}$
\end{lemma}
\begin{pf}
We apply mathematical induction. Assume the statement holds for $ n-1 < l$, we now prove it for \( n \),
\begin{align*}
\partial_t \partial_s^n k &= \partial_t \partial_s(\partial_s^{n-1} k ) \\
&= \partial_s \partial_t(\partial_s^{n-1} k ) - k\left(\partial_s^2 k + \frac{\int (\partial_s k)^2 ds}{\theta}\right)\partial_s^n k \\
&=\partial_s\left(-\partial_s^{n+3}k + \sum_{q'+r'+m'=n-l} c_{q'r'm'}\partial_s^{q'+2} k \partial_s^{r'} k \partial_s^{m'} k + h(t)\sum_{q'+r'=n-l} c_{q'r'} \partial_s^{q'} k \partial_s^{r'} k\right)-k\partial_s^2k\partial_s^n k-h(t)k\partial_s^n k\\
&= -\partial_s^{l+4} k + \sum_{q+r+m=n} c_{qrm} \partial_s^{q+2} k \partial_s^r k \partial_s^m k + h(t)\sum_{q+r=n} c_{qr} \partial_s^{q} k \partial_s^r k.
\end{align*}
\end{pf}

\section{Fundamental Properties of Local-in-Time Solutions of (4)}
\noindent The existence of local-in-time solutions of (4) can be proved by using standard procedures. Indeed, since $\eta_1$ and $\eta_2$ are two skewed straight lines, we can prove the existence of local-in-time solutions of (4) as in [2]. In this section, we collect some basic properties of local-in-time solutions of (4).
\begin{lemma}
Let $\gamma : (-1, 1) \times [0, T) \to \mathbb{R}^2$ be a solution of (4). Then
\[
\int_\gamma k \, ds = \int_{\gamma_0} k_0 \, ds = \theta
\]
for all $t \in [0, T)$, where $k_0$ denotes the curvature of initial curve $\gamma_0$.
\end{lemma}

\begin{pf}
Differentiating the boundary condition
\[
\langle \boldsymbol{v}(-1, t), \boldsymbol{v}_{\eta_0} \rangle = 0
\]
with respect to $t$, we deduce from Lemma 2.5 that
\[
0 = \langle \partial_t \boldsymbol{v}(-1, t), \boldsymbol{v}_{\eta_0} \rangle + \langle \boldsymbol{v}(-1, t), \partial_t \boldsymbol{v}_{\eta_0} \rangle =  \langle \partial_s^3 k(-1, t) \boldsymbol{\tau}(-1, t), \boldsymbol{v}_{\eta_0} \rangle = \partial_s^3 k(-1, t) \langle \boldsymbol{\tau}(-1, t), \boldsymbol{v}_{\eta_0} \rangle = \partial_s^3 k(-1, t).
\]
Similarly, we obtain $\partial_s^3 k(1, t) = 0$. This together with Lemmas 2.4 and 2.5 implies that
\[
\begin{split}
\frac{d}{dt} \int_\gamma k \, ds &= \int_\gamma \left[ \partial_t k + k^2 \partial_s^2 k + k^2h(t)\right] ds = -\int_\gamma \partial_s^4 k \, ds = -\left[ \partial_s^3 k(u, t) \right]_{u=-1}^{u=1} = 0.
\end{split}
\]
This completes the proof.
\end{pf}

In the proof of Lemma 3.1 we get the boundary condition on $\partial_s^3 k$. Thanks to Lemma 2.6 we find the boundary condition on higher order derivatives of $k$.
\begin{lemma}
Let $\gamma : (-1, 1) \times [0, T) \to \mathbb{R}^2$ be a solution of (4). Then
\begin{equation}
\frac{d}{dt} L[\gamma(t)] = 0
\end{equation}
\begin{equation}
\frac{d}{dt} A[\gamma(t)] = h(t)L_0 = \frac{L_0}{\theta}\int_\gamma (\partial_s k)^2 ds,
\end{equation}
for all $t \in [0, T)$.
\end{lemma}

\begin{pf}
Recalling that solutions of (4) are smooth, we observe from (11), Lemma3.1 and the boundary condition in (4) that
\begin{equation*}
\frac{d}{dt} L[\gamma(t)] = \int_\gamma \partial_t ds = \int_\gamma k \partial_s^2 k + h(t)k\, ds = -\int_\gamma (\partial_s k)^2 \, ds + \int_\gamma (\partial_s k)^2 \, ds = 0.
\end{equation*}
Thus (17) follows. Similarly, it follows from (3), Lemmas 2.4 and 2.5 that
\[
\begin{split}
\frac{d}{dt} A[\gamma(t)] &= -\frac{1}{2} \int_\gamma \left[ \langle \partial_t \gamma, \boldsymbol{v} \rangle + \langle \gamma, \partial_t \boldsymbol{v} \rangle \right] ds - \frac{1}{2} \int_\gamma \langle \gamma, \boldsymbol{v} \rangle \partial_t ds \\
&= -\frac{1}{2} \int_\gamma \left[ -\partial_s^2 k - h(t) + \partial_s^3 k \langle \gamma, \boldsymbol{\tau} \rangle \right] ds - \frac{1}{2} \int_\gamma [k \partial_s^2 k + kh(t)]\langle \gamma, \boldsymbol{v} \rangle \, ds.
\end{split}
\]
Integrating by part, we obtain
\[
-\frac{1}{2}\int_\gamma \partial_s^3 k \langle \gamma, \boldsymbol{\tau} \rangle \, ds = \left. -\frac{1}{2}[\partial_s^2 k \langle \gamma, \boldsymbol{\tau} \rangle] \right|_{u=-1}^{u=1} + \frac{1}{2}\int_\gamma \partial_s^2 k \left[ 1 + k \langle \gamma, \boldsymbol{v} \rangle \right] ds.
\]
By using (3), we obtain
\begin{align*}
-\frac{1}{2}h(t)\int_\gamma k\langle \gamma, \boldsymbol{\nu} \rangle \, ds &= -\frac{1}{2}h(t)\int_\gamma \langle \gamma, \partial_s\boldsymbol{\tau} \rangle \, ds = -\frac{1}{2}h(t)[\partial_s(\int_\gamma \langle \gamma, \boldsymbol{\tau}\rangle \,ds)- \int_\gamma\langle\partial_s\gamma, \boldsymbol{\tau}\rangle\, ds]\\
&= -\frac{1}{2}h(t)[\partial_s(\int_\gamma \langle \gamma, \boldsymbol{\tau}\rangle \,ds)- \int_\gamma\langle\boldsymbol{\tau}, \boldsymbol{\tau}\rangle\, ds]\\
&= -\frac{1}{2}h(t)L_0.
\end{align*}
we see that
\begin{equation*}
\frac{d}{dt} A[\gamma(t)] = \left. -\frac{1}{2}[\partial_s^2 k \langle \gamma, \boldsymbol{\tau} \rangle] \right|_{u=-1}^{u=1} + \int_\gamma\partial_s^2k\, ds + h(t)L_0.
\end{equation*}
From the boundary conditions $\partial_s k(-1, t) = 0$ and $\partial_s k(1, t) = 0$, it follows that for all $t \in [0, T)$, we have
\[
\int_\gamma \partial_s^2 k \, ds = \partial_s k(1, t) - \partial_s k(-1, t) = 0.
\]
Recalling that $\eta_0 \cap \eta_\theta = O$, where $O$ denotes the origin of $\mathbb{R}^2$, we see that the boundary conditions $\gamma(-1, t) \in \eta_0$, $\gamma(1, t) \in \eta_\theta$, and $\langle \boldsymbol{v}(-1, t), \boldsymbol{v}_{\eta_0} \rangle = \langle \boldsymbol{v}(1, t), \boldsymbol{v}_{\eta_\theta} \rangle = 0$ imply that $\langle \gamma(1, t), \boldsymbol{\tau}(1, t) \rangle = \langle \gamma(-1, t), \boldsymbol{\tau}(-1, t) \rangle = 0$. This clearly gives
\[
\frac{d}{dt} A[\gamma(t)] = \frac{L_0}{\theta}\int_\gamma (\partial_s k)^2 ds > 0,
\]
and we complete the proof.
\end{pf}

\begin{lemma}
Let $\gamma : (-1, 1) \times [0, T) \to \mathbb{R}^2$ be a solution of (4), then for every  $l \in \mathbb{N} \cup \{0\}$, we have
\[
\partial_s^{2l+1} k(-1, t) = \partial_s^{2l+1} k(1, t) = 0
\]
for all $t \in [0, T)$.
\end{lemma}

\begin{pf}
We proceed by mathematical induction.For odd integer $l \leq 2n-1$, the partial derivative $\partial_s^l k$ vanishes at the boundary(i.e., $\partial_s^l k(-1, t) = \partial_s^l k(1, t)$. Now take $l = 2n-3$.Then we have (evaluating at the boundary)
\[
0 = \partial_t \partial_s^{2n-3} k(\pm 1, t) = -\partial_s^{2n+1} k(\pm 1, t) + \sum_{\substack{q + r + m = 2n-3}} c_{qrm} \partial_s^{q+2} k \partial_s^r k \partial_s^m k(\pm 1, t) + h(t) \sum_{\substack{q + r = 2n-3}} c_{qr} \partial_s^q k \partial_s^r k(\pm 1, t),
\]\\
and then,
\[
\partial_s^{2n+1} k(\pm 1, t) = \sum_{\substack{q + r + m = 2n-3}} c_{qrm} \partial_s^{q+2} k \partial_s^r k \partial_s^m k (\pm 1, t)+ h(t) \sum_{\substack{q + r = 2n-3}} c_{qr} \partial_s^q k \partial_s^r k(\pm 1, t).
\]
Since \( q + 2 + r + m = 2n - 1 \), at least one of \( q + 2 \), \( r \), and \( m \) is an odd integer, and all of them are less than or equal to \( 2n - 1 \), therefore, the first term on the right-hand side equals zero. Similarly, since \( q + r = 2n - 3 \), \( q \) and \( r \) must be one odd and one even.

Therefore, all terms in the sums vanish at the boundary, implying the boundary values of $\partial_s^{2n+1} k$ are time-invariant. Since they are zero in the base case, we conclude:
\[
\partial_s^{2n+1} k(-1, t) = \partial_s^{2n+1} k(1, t) = 0
\]
for all $t \in [0, T)$. This completes the proof.
\end{pf}

\section{Existence of Global-in-Time Solutions}
\noindent Tn this section, we prove the existence of global-in-time solution of (4). To do this, we make use of the isoperimetric inequality for open curves shown by Fuya Hiroi\cite{1}.
\begin{lemma}\cite{1}
Let $\theta \in (0, \pi]$. Let
\[
\tilde{X}_\theta := \left\{ \gamma \in C^1\left([0, 1]; \mathbb{R}^2\right) \mid \gamma(0) \in \eta_0,\ \gamma(1) \in \eta_\theta, \right. \\
\left. \langle v(0), v_{\eta_0} \rangle = \langle v(1), v_{\eta_\theta} \rangle = 0,\ L[\gamma] > 0 \right\}.
\]
For any $\gamma \in \tilde{X}_\theta$, the following isoperimetric inequality holds:
\[
L[\gamma]^2 \geq 2\theta A[\gamma].
\]
The equality is attained if and only if $\gamma$ is a counterclockwise circular arc of central angle $\theta$.
\end{lemma}
\begin{lemma}
Let $\gamma : (-1, 1) \times [0, T) \to \mathbb{R}^2$ be a solution of (4). Then
\begin{align*}
\frac{d}{dt} \operatorname{Kosc}[\gamma] &= -2L_0 \|\partial_s^2 k\|_2^2 + 3L_0 \int_\gamma (k - \bar{k})^2 (\partial_s k)^2 \, ds + 6\bar{k}L_0 \int_\gamma (k - \bar{k})(\partial_s k)^2 \, ds + 2\bar{k}^2 L_0 \|\partial_s k\|_2^2 \\
&- h(t)L_0\left[ \int_\gamma (k - \bar{k})^3\, ds + 3\bar{k}(k - \bar{k})^2\, ds \right]
\end{align*}
for all \( t \in [0, T) \).
\begin{pf}
First, by Lemma3.1 and Lemma3.2 we have
\begin{equation*}
\frac{d}{dt} \bar{k} = \frac{d}{dt}\frac{\int_\gamma k\, ds}{L} = -\frac{\frac{dL}{dt}}{L^2}\int_\gamma k\, ds = 0.
\end{equation*}
This together with  Lemmas 2.4, 2.5, and 3.2 implies that
\begin{align*}
\frac{d}{dt} \text{Kosc}[\gamma] &= \frac{d}{dt} L[\gamma] \int_\gamma (k - \bar{k})^2 \, ds + L_0 \int_\gamma \partial_t (k - \bar{k})^2 \, ds + L_0 \int_\gamma (k - \bar{k})^2 \partial_t \, ds \\
&=-2L_0\int_\gamma \partial_s^4k(k - \bar{k})\, ds - 2L_0\int_\gamma k^2\partial_s^2k(k - \bar{k})\, ds - 2h(t)L_0\int_\gamma k^2(k - \bar{k})\, ds \\
&+ L_0\int_\gamma k\partial_s^2 k (k - \bar{k})^2\, ds+ h(t)L_0 \int_\gamma k (k - \bar{k})^2\, ds.
\end{align*}
Integrating by parts, we deduce from Lemma3.3 that
\[
L[\gamma] \int_\gamma (k - \bar{k})(-\partial_s^4 k) \, ds = -L[\gamma] \|\partial_s^2 k\|_2^2,
\]
\[
L[\gamma] \int_\gamma (k - \bar{k})(-k^2 \partial_s^2 k) \, ds = L[\gamma] \int_\gamma k^2 (\partial_s k)^2 \, ds + 2L[\gamma] \int_\gamma k (k - \bar{k}) (\partial_s k)^2 \, ds,
\]
\[
\begin{split}
L[\gamma] \int_\gamma k (k - \bar{k})^2 \partial_s^2 k \, ds &= -L[\gamma] \int_\gamma (k - \bar{k})^2 (\partial_s k)^2 \, ds - 2L[\gamma] \int_\gamma k (k - \bar{k}) (\partial_s k)^2 \, ds,
\end{split}
\]
and then,
\begin{align*}
\frac{d}{dt} \text{Kosc}[\gamma] &= -2L_0 \|\partial_s^2 k\|_2^2 + 4L_0\int_\gamma k^2(\partial_s k)^2\, ds - 2L_0\bar{k}\int_\gamma k(\partial_s k)^2\, ds - L_0\int_\gamma (\partial_s k)^2 \bar{k})^2\, ds\\
&-2h(t)L_0\int_\gamma k^2(k - \bar{k})\, ds + h(t)L_0 \int_\gamma k (k - \bar{k})^2\, ds\\
&= -2L_0 \|\partial_s^2 k\|_2^2 + 3L_0 \int_\gamma (k - \bar{k})^2 (\partial_s k)^2 \, ds + 6\bar{k}L_0 \int_\gamma (k - \bar{k})(\partial_s k)^2 \, ds + 2\bar{k}^2 L_0 \|\partial_s k\|_2^2\\
& -2h(t)L_0\int_\gamma k^2(k - \bar{k})\, ds + h(t)L_0 \int_\gamma k (k - \bar{k})^2\, ds\\
&= -2L_0 \|\partial_s^2 k\|_2^2 + 3L_0 \int_\gamma (k - \bar{k})^2 (\partial_s k)^2 \, ds + 6\bar{k}L_0 \int_\gamma (k - \bar{k})(\partial_s k)^2 \, ds + 2\bar{k}^2 L_0 \|\partial_s k\|_2^2 \\
&- h(t)L_0\left[ \int_\gamma (k - \bar{k})^3\, ds + 3\bar{k}(k - \bar{k})^2\, ds \right],
\end{align*}
where we used \( \int_\gamma [k - \bar{k}] \, ds = 0 \). Thus we complete the proof.
\end{pf}
\end{lemma}
Using Lemma~4.2, we show an uniform estimate of $\text{Kosc}[\gamma]$.
\begin{lemma}
Let $\gamma : (-1, 1) \times [0, T) \to \mathbb{R}^2$ be a solution of (4) with initial data $\gamma_0$ satisfying (5) and (6). Then
\[
\operatorname{Kosc}[\gamma] < 2K^* \quad \text{for all } t \in [0, T).
\]
\end{lemma}
\begin{pf}
It follows from \cite{1} that,
\begin{equation*}
\left| 3L_0 \int_\gamma (k - \bar{k})^2 (\partial_s k)^2 ds \right| \leq \frac{3}{\pi} \text{Kosc}[\gamma] L_0 |\partial_s^2 k|_2^2,
\end{equation*}
\begin{equation*}
\left| 6\bar{k}L_0 \int_\gamma (k - \bar{k})(\partial_s k)^2 ds \right| \leq \frac{6\theta}{\pi} \sqrt{\text{Kosc}[\gamma]} L_0 \|\partial_s^2 k\|_2^2,
\end{equation*}
Using Lemma~2.2 and 3.3, and H\"{o}lder's inequality we have
\begin{align*}
\int _\gamma (\partial_sk)^2\, ds &= -\int_\gamma k\partial_s^2k\, ds =  -\int_\gamma(k - \bar{k})\partial_s^2k  \leq \left( \int_\gamma(k - \bar{k})^2\,ds \right)^{\frac{1}{2}} \left( \int _\gamma (\partial_s^2k)^2\, ds \right)^{\frac{1}{2}}
\end{align*}
\begin{align*}
\int_\gamma |k - \bar{k}|^3 \, ds &\leq \|k - \bar{k}\|_{L^\infty} \int_\gamma (k - \bar{k})^2 \, ds \leq \sqrt{\frac{L_0}{\pi}} \|\partial_sk\|_2 \int_\gamma (k - \bar{k})^2 \, ds \\
&\leq \frac{L_0}{\pi} \sqrt{\frac{L_0}{\pi}} \|\partial_s^2k\|_2 \int_\gamma (k - \bar{k})^2 \, ds.
\end{align*}\\
Thus,
\begin{align*}
-h(t)L_0 \int_\gamma (k - \bar{k})^3 \, ds &\leq \frac{\int_\gamma k_s^2 \, ds}{\theta} L_0 \int_\gamma |k - \bar{k}|^3 \, ds \\
&\leq \frac{L_0}{\theta} \left( \int_\gamma (k - \bar{k})^2 \, ds \right)^{\frac{1}{2}} \left( \int_\gamma (\partial_s^2k)^2 \, ds \right)^{\frac{1}{2}} \cdot \frac{L_0}{\pi} \sqrt{\frac{L_0}{\pi}} \|\partial_s^2k\|_2 \int_\gamma (k - \bar{k})^2 \, ds \\
&= \frac{L_0^{\frac{5}{2}}}{\theta \pi^{\frac{5}{2}}} \|\partial_s^2k\|_2^2 \left( \int_\gamma (k - \bar{k})^2 \, ds \right)^{\frac{3}{2}} = \frac{L_0}{\theta \pi^{\frac{3}{2}}} \|\partial_s^2k\|_2^2 \text{Kosc}^{\frac{3}{2}}
\end{align*}
By Lemmas 4.2 and the above inequality, we obtain
\begin{equation}
\frac{d}{dt} \text{Kosc} + L_0\left(2 - \frac{L_0}{\theta \pi^{\frac{3}{2}}} \text{Kosc}^{\frac{3}{2}} - \frac{3}{\pi}\text{Kosc} - \frac{6\theta}{\pi}\sqrt{\text{Kosc}}\right) \|\partial_s^2k\|_2^2 \leq  2\bar{k}^2 L_0 \|\partial_s k\|_2^2 - 3h(t)L_0\bar{k}\int_\gamma(k- \bar{k})^2\, ds
\end{equation}
We will prove Lemma~4.3 by contradiction. Suppose there exists \( T^* \in [0, T) \) such that
\[
\text{Kosc}[\gamma(t)] < 2K^* \quad \text{for all } t \in [0, T^*), \quad \text{Kosc}[\gamma(T^*)] = 2K^*.
\]

\noindent Since \( \text{Kosc}[\gamma] < 2K^* \) implies that
\[
2 - \frac{L_0}{\theta \pi^{\frac{3}{2}}} \text{Kosc}^{\frac{3}{2}} - \frac{3}{\pi}\text{Kosc} - \frac{6\theta}{\pi}\sqrt{\text{Kosc}} > 0,
\]
we observe from (6), (19), Lemmas~3.1, and 3.2 that
\[
\begin{split}
\frac{d}{dt} \text{Kosc}[\gamma] &\leq 2\bar{k}^2 L[\gamma] \|\partial_s k\|_2^2 = \frac{2\theta^3}{L_0^2} \frac{dA}{dt}
\end{split}
\]
for all \( t \in [0, T^*] \). This clearly implies that
\[
\text{Kosc}[\gamma(t)] \leq \text{Kosc}[\gamma_0] + \frac{2\theta^3}{L_0^2}(A(t) - A_0) \quad \text{for all } t \in [0, T^*].
\]
 Then it follows from Lemmas~3.2 and 4.1 that
\begin{equation*}
\begin{split}
\text{Kosc}[\gamma(t)] &\leq \text{Kosc}[\gamma_0] + \frac{2\theta^3}{L_0^2}\left(\frac{L_0^2}{2\theta} - A_0\right) \\
&= \text{Kosc}[\gamma_0] + \theta^2\left(1 - \frac{\theta}{2\pi}\cdot \frac{1}{I[\gamma_0]}\right) < K^* + K^* = 2K^*
\end{split}
\end{equation*}
for all \( t \in [0, T^*] \), where we used (6) in the last inequality. However, this contradicts \( \text{Kosc}[\gamma(T^*)] = 2K^* \). Therefore, Lemma~4.3 follows.
\hfill
\end{pf}
Now we prove the existence of global-in-time solution of (4).\\
\begin{theorem}
Let $\theta \in (0, \pi]$. Assume that $\gamma_0$ satisfies (5) and (6). Then problem (4) possesses a unique global-in-time solution.
\end{theorem}
\begin{pf}
Suppose that $T_{\max} < \infty$. It follows from Lemma 4.3 that $\text{Kosc}[\gamma] < 2K^*$ for all $t \in [0, T_{\max})$.Since (6) and Lemma 3.1 imply that $\int_\gamma k \, ds = \theta$ for all $t \in [0, T_{\max})$, we see that
\begin{equation*}
\begin{split}
\text{Kosc}[\gamma]&= L_0\int_\gamma (k - \bar{k})^2 ds = L_0\|k\|_2^2 - 2L_0\bar{k}\int_\gamma k ds + L_0\int \bar{k}^2 ds\\
&= L_0\|k\|_2^2 - \theta^2
\end{split}
\end{equation*}
According to the [4, Theorem 4.1], $\int_\gamma k^2 ds \geq C(T-t)^{-1/4}$, as $t \to T$, $\|k\|_2^2 \to \infty$, then $\|k\|_2 \to \infty$, this contrsdicts Lemma 4.2.
\end{pf}

\section{Full Limit Convergence}
In Sect. 4, we proved the existence of the global-in-time solution to (CD) with $\gamma_0$ satisfying (5) and (6). In this section, we prove the full limit convergence of the solutions to an equilibrium. First, we prepare an interpolation inequality for $k - \bar{k}$.

\begin{lemma}\cite{1}
Let $\gamma$ be a global-in-time solution of (4) with initial data $\gamma_0$ satisfying (5) and (6). Then for all $m \in \mathbb{N}$, $p \geq 2$, and $0 \leq j < m$, we have
\end{lemma}
\begin{equation*}
\|\partial_s^j (k - \bar{k})\|_p \leq M \|k - \bar{k}\|_{2}^{1-\alpha} \|\partial_s^m k\|{2}^\alpha
\end{equation*}
with $M = M(j, m, p) > 0$ and
\[
  \alpha = \frac{1}{m} \left( j + \frac{1}{2} - \frac{1}{p} \right).
\]

\begin{lemma}
Let $\gamma$ be a global-in-time solution of (4). For each $q, r, m, l \in \mathbb{N} \cup \{0\}$
with $q + r + m = l$ and constants $c_{qrm}$, there exist positive constant $C$ and $\alpha_{qrm}$ such that
\[
\sum_{q+r+m=l}c_{qrm} \int_\gamma \partial_s^l k \, \partial_s^{q+2} k \, \partial_s^r k \, \partial_s^m k \, ds
\leq \frac{1}{2}\|\partial_s^{l+2} k\|_2^2 + C_1 \sum_{(q+r+m=l} \operatorname{Kosc}[\gamma]^{\alpha_{qrm}}.
\]
\[
h(t)\sum_{q+r=l}c_{qr} \int_\gamma \partial_s^l k \, \partial_s^q k \, \partial_s^r k \, ds
\leq \frac{1}{2}\|\partial_s^{l+2} k\|_2^2 + C_2 \sum_{q+r=l} \operatorname{Kosc}[\gamma]^{\alpha_{qr}}.
\]
\end{lemma}
\begin{pf}
 The first inequality can be justified by Lemma 5.1 in [1], and it can be proven by taking $\varepsilon$ sufficiently small. We now prove the second inequality.We divide the proof into five steps.\\
\textbf{Step 1.}By Lemma 5.1, we have,
\begin{equation}
\begin{split}
h(t) &= \frac{\int_\gamma (\partial_s k)^2 ds}{\theta} = \frac{1}{\theta} \|\partial_s (k-\bar{k})\|_2^2 \\
&\leq C \|k-\bar{k}\|_2^{2(1-\alpha_0)} \|\partial_s^{l+2} k\|_2^{2\alpha_0}
\end{split}
\end{equation}
 where $\alpha_0=\frac{1}{l+2}$.\\
\textbf{Step 2.} We consider the case $q \neq 0, r \neq 0$. It follows from Holders's inequality, (20) and Lemma 5.1 that
 \begin{equation*}
\begin{split}
h(t)C_{qr} \int_\gamma \partial_s^l k \, \partial_s^q k \, \partial_s^r k \, ds &\leq h(t) C_{qr} \|\partial_s^l k\|_2 \|\partial_s^q k\|_4 \|\partial_s^r k\|_4, \\
&\leq C \|k - \bar{k}\|_2^{5-\alpha_1} \|\partial_s^{l+2} k\|_2^{\alpha_1}
\end{split}
\end{equation*}
where $\alpha_1 = \frac{2l + \frac{5}{2}}{l+2}$.By Young's inequality and lemma 3.2 and 4.1, we have
\begin{align*}
C \|k - \bar{k}\|_2^{5-\alpha_1} \|\partial_s^{l+2} k\|_2^{\alpha_1}
&\leq \varepsilon \|\partial_s^{l+2} k\|_2^2 + C \|k - \bar{k}\|_2^{\frac{2(2l+15)}{3}} \\
&\leq \varepsilon \|\partial_s^{l+2} k\|_2^2 + C \operatorname{Kosc}[\gamma]^{2l+5}.
\end{align*}
\textbf{Step 3.} We consider the case $q=0, r \neq 0$ or $q \neq 0, r=0$. It is sufficient to consider the case of $q=0, r \neq 0$.
\begin{equation*}
\begin{split}
h(t)C_{qr} \int_\gamma \partial_s^l k \, \partial_s^r k \, k \, ds &= h(t)C_{qr} \int_\gamma (\partial_s^l k)^2 \, k \, ds \\
&= h(t)C_{qr} \int_\gamma (\partial_s^l k)^2 \, (k - \bar{k}) \, ds + h(t)C_{qr} \bar{k} \int_\gamma (\partial_s^l k)^2 \, ds \\
&=: I_2 + I_3.
\end{split}
\end{equation*}
In the same way as in Step 2, it follows from Holders's inequality, (20) and Lemma 5.1 that
\begin{align*}
|I_2| &\leq h(t)C_{qr} \|\partial_s^{l} k\|_2 \|\partial_s^l k\|_4 \|k - \bar{k}\|_4 \leq C \|k - \bar{k}\|_2^{5-\alpha_2} \|\partial_s^{l+2} k\|_2^{\alpha_2}, \\
|I_3| &= h(t)C_{qr} \bar{k}\|\partial_s^{l} (k-\bar{k})\|_2^2 \|\leq C \|k - \bar{k}\|_2^{4-\alpha_3} \|\partial_s^{l+2} k\|_2^{\alpha_3},
\end{align*}
where $\alpha_2 = (2l + \frac{5}{2})/(l + 2)$ and $\alpha_3 = (2l + 2)/(l + 2)$. This together with Young's inequality implies that
\begin{align*}
|I_1| &\leq \varepsilon \|\partial_s^{l+2} k\|_2^2 + C \operatorname{Kosc}[\gamma]^{2l+5}, \\
|I_2| &\leq \varepsilon \|\partial_s^{l+2} k\|_2^2 + C \operatorname{Kosc}[\gamma]^{l+3}.
\end{align*}
Thus we obtain
\[
h(t)C_{qr} \int_\gamma (\partial_s^l k)^2 \, k \, ds \leq \varepsilon \|\partial_s^{l+2} k\|_2^2 + C \left[ \operatorname{Kosc}[\gamma]^{2l+5} + \operatorname{Kosc}[\gamma]^{l+3} \right].
\]
\textbf{Step 4.} We consider the case $q=r=0$.
\begin{align*}
h(t)C_{qr} \int_\gamma \partial_s^l k \, \partial_s^q k \, \partial_s^r k \, ds
&= h(t)C_{qr} \int_\gamma \partial_s^l k k^2 \, ds \\
&= h(t)C_{qr} \int_\gamma \partial_s^l k [(k-\bar{k})^2+\bar{k}]^2 \, ds \\
&= h(t)C_{qr} \int_\gamma \partial_s^l k \cdot (k - \bar{k})^2 \, ds + 2h(t)C_{qr} \bar{k} \int_\gamma \partial_s^l k \cdot (k - \bar{k}) \, ds + h(t)C_{qr} \bar{k}^2 \int_\gamma \partial_s^l (k-\bar{k}) \, ds \\
&=: I_4 + I_5 + I_6.
\end{align*}
 It follows from Holders's inequality, (20) and Lemma 5.1 that
\begin{align*}
|I_4| &\leq h(t)C_{qr}(\int_\gamma(\partial_s^l k)^2 ds)^{\frac{1}{2}}(\int_\gamma(k-\bar{k})^4 ds)^{\frac{1}{2}}\leq C \|k - \bar{k}\|_2^{5-\alpha_4} \|\partial_s^{l+2} k\|_2^{\alpha_4},\\
|I_5| &\leq 2h(t)\bar{k}C_{qr}(\int_\gamma(\partial_s^l k)^2 ds)^{\frac{1}{2}}(\int_\gamma(k-\bar{k})^2 ds)^{\frac{1}{2}}\leq C \|k - \bar{k}\|_2^3 \|\partial_s^{l+2} k\|_2,\\
|I_6|&\leq h(t)\bar{k}^2C_{qr}(\int_\gamma 1^2 ds)^{\frac{1}{2}}(\int_\gamma(\partial_s^l(k-\bar{k}))^2 ds)^{\frac{1}{2}}\leq C \|k - \bar{k}\|_2^2 \|\partial_s^{l+2} k\|_2.
\end{align*}
where $\alpha_2 = (l + \frac{5}{2})/(l + 2)$. This together with Young's inequality implies that
\begin{align*}
|I_4| &\leq \varepsilon \|\partial_s^{l+2} k\|_2^2 + C \operatorname{Kosc}[\gamma]^{\frac{4l+\frac{15}{2}}{l+\frac{3}{2}}}, \\
|I_5| &\leq \varepsilon \|\partial_s^{l+2} k\|_2^2 + C \operatorname{Kosc}[\gamma]^{3},\\
|I_6| &\leq \varepsilon \|\partial_s^{l+2} k\|_2^2 + C \operatorname{Kosc}[\gamma]^{2}
\end{align*}
Thus we obtain
\[
h(t)C_{qr} \int_\gamma \partial_s^l k \, k^2 \, ds \leq \varepsilon \|\partial_s^{l+2} k\|_2^2 + C \left[ \operatorname{Kosc}[\gamma]^{\frac{4l+\frac{15}{2}}{l+\frac{3}{2}}} + \operatorname{Kosc}[\gamma]^{3} + \operatorname{Kosc}[\gamma]^{2} \right].
\]
\textbf{Step 5.} We prove the required inequality. Thanks to Step 1, Step 2, Step 3 and Step 4, taking \(\varepsilon > 0\) small enough, we obtain
\[
h(t)\sum_{(q,r,m)\in\mathcal{N}_l} c_{qrm} \int_\gamma \partial_s^l k \, \partial_s^{q} k \, \partial_s^r k ds \leq \frac{1}{2}\|\partial_s^{l+2} k\|_2^2 + C_2 \sum_{q+r=l} \operatorname{Kosc}[\gamma]^{\alpha_{qr}}.
\]

Therefore, we complete the proof.
\end{pf}
\begin{lemma}
Let $\gamma$ be a global-in-time solution of (4) with initial data $\gamma_0$ satisfying (5) and (6). Then, for each $l \in \mathbb{N}$, there exists a constant $C_l > 0$ such that
\[
\|\partial_s^l k\|_2^2 \leq C_l \quad \text{for all } t \in [0, \infty).
\]
\end{lemma}

\begin{pf}
From Lemmas 2.4, 2.6, 3.3, 4.4 and 5.2, we find a constant $C > 0$ such that
\begin{align*}
\frac{d}{dt} \|\partial_s^l k\|_2^2
&= 2 \int_\gamma \partial_t \partial_s^l k \, \partial_s^l k \, ds + \int_\gamma k\, (\partial_s^l k)^2 \, (\partial_s^2 k + h(t))\, ds \\
&= -2 \int_\gamma \partial_s^l k \, \partial_s^{l+4} k \, ds + \sum_{q+r+m=l} C_{qrm} \int_\gamma \partial_s^l k \, \partial_s^{q+2} k \, \partial_s^r k \, \partial_s^m k \, ds \\
&+ h(t)\sum_{q+r=l} C_{qr} \int_\gamma \partial_s^l k \, \partial_s^{q} k \, \partial_s^r k \, ds
+ \int_\gamma k \partial_s^2 k \, (\partial_s^l k)^2 \, ds\\
&+ h(t)\int_\gamma k\, (\partial_s^l k)^2 \, ds\\
&= -2\|\partial_s^{l+2} k\|_2^2 + \sum_{q+r+m=l} C_{qrm}\int_\gamma \partial_s^l k \, \partial_s^{q+2} k \, \partial_s^r k \, \partial_s^m k \, ds + h(t)\sum_{q+r=l} C_{qr} \int_\gamma \partial_s^l k \, \partial_s^{q} k \, \partial_s^r k \, ds \\
&\leq -\|\partial_s^{l+2} k\|_2^2 + C_1  \sum_{q+r+m=l} \operatorname{Kosc}[\gamma]^{\alpha_{qrm}}  + C_2  \sum_{q+r=l} \operatorname{Kosc}[\gamma]^{\alpha_{qr}} \\
&\leq -\|\partial_s^{l+2} k\|_2^2 + C.
\end{align*}
Then, from Lemmas 2.1, 2.2, 3.2, and 4.1, we find positive constants $C_1$ and $C_2$ such that
\[
\frac{d}{dt} \|\partial_s^l k\|_2^2 \leq -C_1 \|\partial_s^l k\|_2^2 + C_2 \quad \text{for all } t > 0.
\]
This clearly implies that
\[
\|\partial_s^l k(t)\|_2^2 \leq \left[ \|\partial_s^l k(0)\|_2^2 - \frac{C_2}{C_1} \right] e^{-C_1 t} + \frac{C_2}{C_1} \quad \text{for all } t > 0.
\]
Therefore Lemma 5.3 follows.
\end{pf}

Here we prove the decay of $\operatorname{Kosc}[\gamma]$ as $t \to \infty$.
\begin{lemma}
Let $\gamma$ be a global-in-time solution of (4) with initial data $\gamma_0$ satisfying (5) and (6). Then
\[
\lim_{t\to\infty} \operatorname{Kosc}[\gamma] = 0.
\]
\end{lemma}

\begin{pf}
First, we verify that the derivative of $\operatorname{Kosc}[\gamma]$ is bounded in $[0, \infty)$. Indeed, by plugging Lemmas 3.1, 3.2, 4.2, and 5.3 into (19), we find a constant $C > 0$ such that
\[
\frac{d}{dt} \operatorname{Kosc}[\gamma] \leq 2\bar{k}^2 L_0 \|\partial_s k\|_2^2 < C.
\]
In particular, $\operatorname{Kosc}[\gamma]$ is uniformly continuous in $[0, \infty)$.
By Lemmas 2.1 and 3.2, we have
\[
|\operatorname{Kosc}[\gamma]| = L_0 \|k - \bar{k}\|_2^2 \leq \frac{L_0^3}{\pi^2} \|\partial_s k\|_2^2 = \frac{\theta L_0^2}{\pi^2} \frac{dA}{dt}.
\]
Integrating the both sides with respect to $t$, we obtain
\begin{equation}
\int_0^T |\operatorname{Kosc}[\gamma]| dt \leq \frac{\theta L_0^2}{\pi^2}[A(t)-A_0] \leq \frac{L_0^4}{2\pi^2}
\end{equation}
for all $T > 0$.

Therefore, we see that $\operatorname{Kosc}[\gamma]$ is in $L^1(0, \infty)$ and uniformly continuous on $[0, \infty)$. This clearly implies that $\operatorname{Kosc}[\gamma] \to 0$ as $t \to \infty$. We complete the proof.
\end{pf}

Thanks to Lemmas 5.1 and 5.4, we have:
\begin{lemma}
Let $\theta \in (0, \pi)$. Let $\gamma$ be a global-in-time solution of (4) with initial data $\gamma_0$ satisfying (5) and (6). Then there exist positive constants $C$ and $\delta$ such that
\[
\|\partial_s^2 k\|_2^2 < C e^{-\delta t} \quad \text{for all } t \in [0, \infty).
\]
\end{lemma}
\begin{pf}
It follows from Lemmas 2.4, 2.5, 3.3 and integrating by parts, we have
\begin{align*}
\frac{d}{dt} \|\partial_s^2 k\|_2^2
&= 2 \int_\gamma \partial_t \partial_s^2 k \, \partial_s^2 k \, ds + \int_\gamma k (\partial_s^2 k + h(t))(\partial_s^2 k)^2 \, ds \\
&= -2 \int_\gamma \partial_s^6 k \, \partial_s^2 k \, ds - 6 \int_\gamma (\partial_s k)^2 (\partial_s^2 k)^2 \, ds - 8 \int_\gamma k (\partial_s^2 k)^3 \, ds - 10 \int_\gamma k \partial_s k \, \partial_s^2 k \, \partial_s^3 k \, ds\\
 &- 2 \int_\gamma k^2 \partial_s^2 k \, \partial_s^4 k \, ds -2h(t)\int_\gamma (\partial_s^2 k)^2 \partial_s^2 k -4h(t)\int_\gamma k(\partial_s^2 k)^2 ds+\int_\gamma k(\partial_s^2 k)^3 ds+h(t)\int_\gamma k(\partial_s^2 k)^2 ds\\
&= -2 \int_\gamma \partial_s^6 k \, \partial_s^2 k \, ds - 6 \int_\gamma (\partial_s k)^2 (\partial_s^2 k)^2 \, ds - 7 \int_\gamma k (\partial_s^2 k)^3 \, ds - 10 \int_\gamma k \partial_s k \, \partial_s^2 k \, \partial_s^3 k \, ds\\
 &- 2 \int_\gamma k^2 \partial_s^2 k \, \partial_s^4 k \, ds -2h(t)\int_\gamma (\partial_s^2 k)^2 \partial_s^2 k -3h(t)\int_\gamma k(\partial_s^2 k)^2 ds\\
&= -2 \|\partial_s^4 k\|_2^2 + \|\partial_s k \, \partial_s^2 k\|_2^2 + 4 \int_\gamma k \partial_s k \, \partial_s^2 k \, \partial_s^3 k \, ds - 2 \int_\gamma k^2 \partial_s^2 k \, \partial_s^4 k \, ds\\
&-2h(t)\int_\gamma (\partial_s^2 k)^2 \partial_s^2 k -3h(t)\int_\gamma k(\partial_s^2 k)^2 ds.
\end{align*}
More over, since
\begin{align*}
4 \int_\gamma k \partial_s k \, \partial_s^2 k \, \partial_s^3 k \, ds&= 2\int_\gamma \partial_s (k^2) \, \partial_s^2 k \, \partial_s^3 k \, ds \\
&= -2 \|k \partial_s^3 k\|_2^2 - 2 \int_\gamma k^2 \partial_s^2 k \, \partial_s^4 k \, ds,
\end{align*}
it follows that
\begin{align*}
\frac{d}{dt} \|\partial_s^2 k\|_2^2
&= -2 \|\partial_s^4 k\|_2^2 + \|\partial_s k \, \partial_s^2 k\|_2^2 - 2 \|k \partial_s^3 k\|_2^2 - 4 \int_\gamma k^2 \partial_s^2 k \, \partial_s^4 k \, ds - 3h(t)\int_\gamma k(\partial_s^2 k)^2 ds \\
&=: -2 \|\partial_s^4 k\|_2^2 + I_1 + I_2 + I_3 + I_4.
\end{align*}
From [1], it is known that
\begin{align*}
I_1 &\leq \frac{2M^4}{\pi} \operatorname{Kosc}[\gamma] \|\partial_s^4 k\|_2^2, \\
I_2 &\leq \frac{4M^2 \bar{k}L_0}{\pi^\frac{3}{2}} \operatorname{Kosc}[\gamma]^\frac{1}{2} \|\partial_s^4 k\|_2^2 - 2\bar{k}^2 \|\partial_s^3 k\|_2^2, \\
I_3 &\leq \frac{4M^3}{\pi} \operatorname{Kosc}[\gamma] \|\partial_s^4 k\|_2^2 + \frac{8M^2 \bar{k}L_0}{\pi^\frac{3}{2}} \operatorname{Kosc}[\gamma]^\frac{1}{2} \|\partial_s^4 k\|_2^2 + 4\bar{k}^2  \|\partial_s^3 k\|_2^2
\end{align*}
It follows from Holders's inequality and Lemma 5.1 that
\begin{align*}
I_4 &= -3h(t) \int_\gamma (k - \bar{k}) (\partial_s^2 k)^2 \, ds - 3h(t)\bar{k} \int_\gamma (\partial_s^2 k)^2 \, ds \\
&\leq 3h(t) \|\partial _s^2 k \|_2(\int _\gamma(k-\bar{k})^2(\partial_s^2 k)^2 ds)^\frac{1}{2}\\
&\leq \frac{3}{\theta} \|\partial _s k \|_2^2\|\partial _s^2 k \|_2\|\partial _s^2 k \|_4\|k-\bar{k} \|_4 \\
&\leq \frac{3}{\theta}M^4\|k-\bar{k} \|_2^{\frac{48}{16}} \|k-\bar{k} \|_2^{\frac{6}{16}}\|\partial _s^4 k \|_2^{\frac{26}{16}}\\
&\leq \frac{3}{\theta}M^4\|k-\bar{k} \|_2^3(\frac{L_0}{\pi})^{\frac{3}{2}} \|\partial_s^4k \|_2^{\frac{6}{16}}\|\partial _s^4 k \|_2^{\frac{26}{16}}\\
&=\frac{3M^4}{\pi^{\frac{3}{2}}\theta}\operatorname{Kosc}[\gamma]^\frac{3}{2}\|\partial _s^4 k \|_2^2.
\end{align*}

Hence we get
\begin{align*}
\frac{d}{dt} \| \partial_s^2 k \|_2^2
&\leq \left[ -2 + \frac{2M^3(M+2)}{\pi} \operatorname{Kosc}[\gamma] + \frac{12M^2 \bar{k} L_0}{\pi^{\frac{3}{2}}} \operatorname{Kosc}[\gamma]^{\frac{1}{2}}+\frac{3M^4}{\pi^{\frac{3}{2}}\theta}\operatorname{Kosc}[\gamma]^\frac{3}{2} \right] \| \partial_s^4 k \|_2^2 + 2\bar{k}^2 \| \partial_s^3 k \|_2^2 \\
&\leq \left[ -2 + \frac{2M^3(M+2)}{\pi} \operatorname{Kosc}[\gamma] + \frac{12M^2 \bar{k} L_0}{\pi^{\frac{3}{2}}} \operatorname{Kosc}[\gamma]^{\frac{1}{2}}+\frac{3M^4}{\pi^{\frac{3}{2}}\theta}\operatorname{Kosc}[\gamma]^\frac{3}{2} \right] \| \partial_s^4 k \|_2^2 + \frac{2\theta^2}{\pi^2} \| \partial_s^4 k \|_2^2 \\
&= \left[ -2\left( 1 - \frac{\theta^2}{\pi^2} \right) + \frac{2M^3(M+2)}{\pi} \operatorname{Kosc}[\gamma] + \frac{12M^2 \bar{k} L_0}{\pi^{\frac{3}{2}}} \operatorname{Kosc}[\gamma]^{\frac{1}{2}}+\frac{3M^4}{\pi^{\frac{3}{2}}\theta}\operatorname{Kosc}[\gamma]^\frac{3}{2} \right] \| \partial_s^4 k \|_2^2.
\end{align*}
Let $\varepsilon := 1 - \theta^2/\pi^2 > 0$. Thanks to Lemmas 2.1, 2.2, 3.2, 4.1, and 5.4, we find $T > 0$ such that
\[
\frac{d}{dt} \| \partial_s^2 k \|_2^2 \leq -\varepsilon \| \partial_s^4 k \|_2^2 \leq -\frac{\varepsilon \pi^4}{L_0^{4}} \| \partial_s^2 k \|_2^2 \quad \text{for all } \ t \geq T.
\]
This together with Lemmas 3.2, 4.1 and integrating the above inequality yield the conclusion.
\end{pf}
\begin{lemma}
Let $\theta \in (0, \pi)$. Let $\gamma$ be a global-in-time solution of (4) with initial data $\gamma_0$ satisfying (5) and (6). Then, for each $l \in \mathbb{N}$, there exist positive constants $C$ and $\delta$ such that
\[
\|\partial_s^l k\|_{L^\infty} < C e^{-\delta t} \quad \text{for all } t \in [0, \infty).
\]
\end{lemma}

\begin{pf}
It follows from Lemmas 2.2, 3.2, and 5.5 that
\[
\|\partial_s k\|_2 \leq \frac{L_0}{\pi} \|\partial_s^2 k\|_2 \leq \frac{C L_0}{\pi} e^{-\delta t}
\]
for all $t \in [0, \infty)$. This together with Lemmas 2.1 and 3.2 implies that
\begin{equation}
\operatorname{Kosc}[\gamma] = L_0 \|k - \bar{k}\|_2^2 \leq \frac{L_0^3}{\pi^2} \|\partial_s k\|_2^2 \leq \frac{C L_0^4}{\pi^3} e^{-\delta t}
\end{equation}
for all $t \in [0, \infty)$. For each $l \geq 3$, combining (22) with Lemmas 5.1, 5.3 and 5.5, we find $C > 0$ such that
\begin{equation}
\begin{split}
\|\partial_s^l k\|_2 &\leq M \|k - \bar{k}\|_2^{1-\alpha} \|\partial_s^{l+1} k\|_2^\alpha \\
&= \frac{M}{L_0^{\frac{1-\alpha}{2}}} \operatorname{Kosc}[\gamma]^{\frac{1-\alpha}{2}} \|\partial_s^{l+1} k\|_2^\alpha\\
& \leq C e^{-\frac{1-\alpha}{2} \delta t}
\end{split}
\end{equation}
for all $t \in [0, \infty)$. If $l \in \mathbb{N}$ is odd, then Lemma 3.3 gives $\partial_s^l k(-1, t) = \partial_s^l k(1, t) = 0$, and while if $l \in \mathbb{N}$ is even, then Lemma 3.3 also gives $\int_\gamma \partial_s^l k \, ds = 0$. Thus, it follows from (23) and Lemmas 2.1, 2.2, and 3.2 that
\[
\|\partial_s^l k\|_{L^\infty}^2 \leq \frac{2 L_0}{\pi} \|\partial_s^{l+1} k\|_2^2 \leq C e^{-\delta t}
\]
for all $t \in [0, \infty)$ and $l \in \mathbb{N}$. Hence Lemma 5.6 follows.
\end{pf}
\begin{lemma}
Let $\theta \in (0, \pi)$. Let $\gamma$ be a global-in-time solution of (4) with initial curve $\gamma_0$ satisfying (5) and (6). Then there exist positive constants $C_1, C_2$ and $C_3$ such that
\begin{equation}
|\gamma(u, t)| < C_1 \quad \text{for all } (u, t) \in [-1, 1] \times [0, \infty),
\end{equation}
\begin{equation}
C_2 < |\partial_u \gamma(u, t)| < C_3 \quad \text{for all } (u, t) \in [-1, 1] \times [0, \infty).
\end{equation}
Moreover, for each $l \in \mathbb{N}$, there exists a positive constant $C_4$ such that
\begin{equation}
|\partial_u^l \gamma| < C_4 \quad \text{for all } (u, t) \in [-1, 1] \times [0, \infty).
\end{equation}
\end{lemma}
\begin{pf}
First, we prove (24). We observe from Lemma 5.6 that
\begin{align*}
|\gamma(u, \tau) - \gamma_0(u)| &\leq \int_0^\tau |\partial_t\gamma(u,t)|dt=\int_0^\tau |\partial_s^2k(u,t)+\frac{\int_\gamma(\partial_sk)^2ds}{\theta}|dt\\
&\leq \int_0^\tau \|{\partial_s^{2}k(t)}\|_{L^\infty} dt +\int_0^\tau \frac{\int_\gamma\|\partial_s^{2}k(t)\|_{L^\infty}ds}{\theta} dt\\
&\leq C \int_0^\tau e^{-\frac{1}{2}\delta t} dt + C \int_0^\tau e^{-\delta t} dt\\
&= \frac{2C}{\delta} \left[1 - e^{-\frac{1}{2}\delta \tau}\right]+\frac{2C}{\delta} \left[1 - e^{-\delta \tau}\right] < \frac{2C}{\delta}+\frac{C}{\delta}<C_1
\end{align*}
for all $(u, \tau) \in [-1, 1] \times (0, \infty)$. Thus (24) follows.
We turn to (25). From (3) we have
\[
\frac{d}{dt} |\partial_{u}\gamma|^2 = 2\langle \partial_{u}\gamma, \partial_{u}\partial_{t}\gamma \rangle = 2|\partial_{u}\gamma|^2 \langle \boldsymbol{\tau}, -\partial_{s}(\partial_{s}^{2}k+h(t))\boldsymbol{\nu}) \rangle = 2k(\partial_{s}^{2}k+h(t)) |\partial_{u}\gamma|^2,
\]
i.e.,
\[
\frac{d}{dt} |\partial_{u}\gamma| = k(\partial_{s}^{2}k+h(t))|\partial_{u}\gamma|.
\]

This clearly implies that
\begin{equation}
|\partial_{u}\gamma(u, t)| = |\gamma_0'(u)| \exp\left[ \int_0^t k(\partial_{s}^{2}k+h(t)) \, d\tau \right] \quad \text{for all } \ (u, t) \in (-1, 1) \times [0, \infty).
\end{equation}
From Lemmas 2.1, 3.1, 3.2, 4.1, and 5.3, we find $C > 0$ such that
\begin{equation}
\|{k}\|_{L^\infty} \leq \bar{k} + \|{k - \bar{k}}\|_{L^\infty} \leq \frac{\theta}{L_0} + \sqrt{\frac{2L_0}{\pi}} \|{\partial_{s}k}\|_2 \leq C.
\end{equation}
Combining Lemma 5.5 with (27) and (28), we get
\begin{align*}
 |\partial_{u}\gamma(u, t)| &\leq |\gamma_0'| \exp\left[ \int_0^t\|k\|_{L^\infty}\|\partial_{s}^{2}k\|_{L^\infty} \, d\tau + \int_0^t\|k\|_{L^\infty}\frac{\int_\gamma\|\partial_{s}^{2}k\|_{L^\infty}^2ds}{\theta} \, d\tau \right] \\
 &\leq \|\gamma_0'\|_{L^\infty} \exp\left[ \int_0^t\|k\|_{L^\infty}\|\partial_{s}^{2}k\|_{L^\infty} \, d\tau + \frac{L_0}{\theta} \int_0^t \|k\|_{L^\infty}\|\partial_{s}^{2}k\|_{L^\infty} \, d\tau \right]\\
 &\leq \|{\gamma_0'}\|_{L^\infty} \exp\left[ \int_0^\infty C e^{-\frac{1}{2}\delta \tau} d\tau + \int_0^\infty C e^{-\delta \tau} d\tau\right]  < C
\end{align*}
 for all $(u, t) \in (-1, 1) \times [0, \infty)$, and while,
 \begin{align*}
  |\partial_{u}\gamma(u, t)|&\geq |\gamma_0'(u)| \exp\left[- \int_0^t\|k\|_{L^\infty}\|\partial_{s}^{2}k\|_{L^\infty} \, d\tau - \frac{L_0}{\theta} \int_0^t \|k\|_{L^\infty}\|\partial_{s}^{2}k\|_{L^\infty} \, d\tau \right]\\
 &\leq \|{\gamma_0'}\|_{L^\infty} \exp\left[ -\int_0^\infty C e^{-\frac{1}{2}\delta \tau} d\tau - \int_0^\infty C e^{-\delta \tau} d\tau\right]  > C
 \end{align*}
 for all $(u, t) \in (-1, 1) \times [0, \infty)$.
 We turn to (26). For any smooth vector field $\phi$, we obtain
\begin{align}
 \partial_{t}\phi = v^l \partial_{s}\phi - P\left( v, \partial_{u}(v^{-1}), \dots, \partial_{u}^{l-1}(v^{-1}); \partial_{u}\phi, \dots, \partial_{u}^{l-1}\phi \right).
 \end{align}
where \( v := | \partial_u \gamma | \) and \( P \) denotes a polynomial of \( \phi, \partial_u\phi, \dots, \partial_u^{l-1}\phi \) with coefficients \( v, v^{-1}, \partial_u(v^{-1}), \dots, \partial_u^{l-1}(v^{-1}) \). We claim that
\begin{equation}
\begin{split}
\displaystyle \partial_u^l \left[ \|\partial_u \gamma\|_2^{-1} \right] \leq c_l \quad \text{for all } (u, t) \in [-1, 1] \times [0, \infty), \\[6pt]
\displaystyle \|\partial_u^l \left[ k \partial_s^2 k \right]\|_{L^\infty} \leq c_l e^{-\delta t} \quad \text{for all } (u, t) \in [-1, 1] \times [0, \infty),
 \end{split}
\end{equation}
for all \( l \in \mathbb{N} \cup \{0\} \). We have already shown that (30) holds true for \( l = 0 \). Suppose that (30) holds true up to \( m \in \mathbb{N} \cup \{0\} \). Since
\begin{equation*}
\begin{split}
\partial_t \partial_u^{m+1} \left( |\partial_u \gamma|^{-1} \right) &= -\partial_u^{m+1} \left[ |\partial_u \gamma|^{-2} \partial_t|\partial_u \gamma |\right]= -\partial_u^{m+1} \left[|\partial_u \gamma|^{-1} k (\partial_s^2 k+h(t)) \right]\\
&= -\partial_u^{m+1} \left[ |\partial_u \gamma|_2^{-1} \right] k (\partial_s^2 k+h(t)) - \sum_{k=0}^m \binom{m+1}{k} \partial_u^k \left[ |\partial_u \gamma|^{-1} \right] \partial_u^{m+1-k} \left[ k \partial_s^2 k \right]\\
&-h(t)\sum_{k=0}^m \binom{m+1}{k} \partial_u^k \left[ |\partial_u \gamma|^{-1} \right] \partial_u^{m+1-k}k,
\end{split}
\end{equation*}
we find \( C > 0 \) such that
\begin{equation*}
\partial_t \partial_u^{m+1} \left( |\partial_u \gamma|^{-1} \right) + \partial_u^{m+1} \left[ |\partial_u \gamma|^{-1} \right] k (\partial_s^2 k+h(t)) \leq C \sum_{k=0}^m \binom{m+1}{k} |\partial_u^{m+1-k} \left[ k \partial_s^2 k \right]|+h(t)C\sum_{k=0}^m \binom{m+1}{k} |\partial_u^{m+1-k} k|.
\end{equation*}
The differential inequality implies that
\begin{equation}
\begin{split}
\frac{d}{dt} \left[ \partial_u^{m+1} \left( |\partial_u \gamma|^{-1} \right) \exp\left[ \int_0^t k (\partial_s^2 k+h(t)) d\tau \right] \right]
&\leq C \sum_{k=0}^m \binom{m+1}{k} |\partial_u^{m+1-k} \left[ k\partial_s^2 k  \right]| \exp\left[ \int_0^t k (\partial_s^2 k+h(t)) d\tau \right]\\
&+h(t)C\sum_{k=0}^m \binom{m+1}{k} |\partial_u^{m+1-k}  k| \exp\left[ \int_0^t k (\partial_s^2 k+h(t)) d\tau \right].
\end{split}
\end{equation}
By Lemma 5.6, we have
\begin{equation}
h(t)=\frac{1}{\theta}\int_\gamma\partial_s^2k ds\leq \frac{L_0}{\theta}\|\partial_{s}^{2}k\|_{L^\infty}\leq C
\end{equation}
Taking \( \phi = k \partial_s^2 k \) in (29) with \( l = m + 1 \), we observe from the inductive assumption, (25), and Lemma 5.6 that
\begin{equation}
|\partial_u^{m+1} \left[ k \partial_s^2 k \right]| \leq c_{m+1} e^{-\delta t} \quad \text{for all } (u, t) \in [-1, 1] \times [0, \infty).
\end{equation}
Combining (31) and (32) we obtain
\[
\partial_u^{m+1} \left( |\partial_u \gamma(u, \tau)|^{-1} \right) \leq c_{m+1} \quad \text{for all } \ (u, t) \in [-1, 1] \times [0, \infty).
\]
Thus we see that (30) holds. We prove (26) by induction. Taking \( \phi = \gamma \) and \( l = 2 \) in (29), we see from (3) that
\[
\partial_u^2 \gamma = v^2 k v - P\left( v, \partial_u(v^{-1}); \partial_u \gamma \right).
\]
This together with (26), (29), and (31) implies that
\[
\left| \partial_u^2 \gamma(u, t) \right| \leq C \quad \text{for all } \ (u, t) \in [-1, 1] \times [0, \infty).
\]
Suppose that \( |\partial_u^k \gamma(u, t)| < C \) for all \( (u,t) \in [-1,1] \times [0, \infty) \) for all \( k \in \{2, \dots, m\} \) with some \( m \geq 2 \). Taking \( \phi = \gamma \) and \( l = m + 1 \) in (29), we have
\[
\partial_u^{m+1} \gamma = v^{m+1} \partial_s^{m-1}(kv) - P\left( v, \partial_u(v^{-1}), \dots, \partial_u^m(v^{-1}); \partial_u \gamma, \dots, \partial_u^m \gamma \right).
\]
Combining this with (30), Lemma 5.6, and the inductive assumption, we find \( C > 0 \) such that
\[
|\partial_u^{m+1} \gamma(u,t)| \leq C \quad \text{for all } \ (u,t) \in [-1,1] \times [0, \infty).
\]
This completes the proof.

We conclude this paper by proving full
limit convergence of the solutions of(4).
\end{pf}
\begin{theorem}
Let \( \theta \in (0, \pi) \) and \( \gamma_0 \) satisfy (5) and (6). Then the unique global-in-time solution of (4) converges exponentially to the circular arc \( \gamma_\infty \) of central angle \( \theta \) and length \( L[\gamma_\infty] = L[\gamma_0] \) as \( t \to \infty \) in the \( C^\infty \)-topology.
\end{theorem}
\begin{pf}
The proof adopts the core framework of \cite{1}, with the primary modification being the substitution of the evolution equation for the curve diffusion flow (4) into Step 1 (subconvergence analysis) and Step 4 (convergence rate estimation). All critical integral and differential estimates are derived from Lemmas 5.5 and 5.6 in the preceding sections, and the fundamental convergence logic of \cite{1} is retained. The proof proceeds in four steps:\\
\textbf{Step 1.}
By Lemma 5.5, 5.6 (with the evolution equation substituted for the time derivative term), we have:
\begin{equation*}
\int_0^\tau \|\partial_t\gamma\|_2^2 dt =  \int_0^\tau \| (\partial_s^2 k + h(t))\boldsymbol{\nu} \|_2^2 dt \leq C\int_0^\tau e^{-\delta t} dt = \frac{C}{\delta}[1 - e^{-\delta\tau}]
\end{equation*}
implying $\int_0^\infty \|\partial_t\gamma\|_2 dt < \infty$. A monotone divergent sequence $\{t_j\}$ is constructed such that $\|\partial_t \gamma(t_j)\|_2 \to 0$ as $j \to \infty$. By the Arzel\`{a}-Ascoli theorem and diagonal argument, a subsequence $\{t_{j(k)}\}$ yields $\gamma(\cdot, t_{j(k)}) \to \gamma_\infty \in C^\infty([-1,1])$ ($C^\infty$-topology), and $\gamma_\infty$ is an equilibrium of (4).\\
\textbf{Step 2.}
By contradiction (following \cite{1}), assume two sequences $\{t_j\} \ and \ \{t_k'\}$ satisfy $\gamma(\cdot, t_j) \to \gamma_\infty$ and $\gamma(\cdot, t_k') \to \tilde{\gamma}_\infty \neq \gamma_\infty$. Define $G[\gamma(t)] = \int_{-1}^1 |{\gamma(u,t) - \gamma_\infty(u)}|^2 du$, by Lemma 5.7 gives $|\frac{d}{dt}G[\gamma(t)]| \leq C\|\partial_t \gamma\|_2$. Combining with $\int_0^\infty \|\partial_t \gamma\|_2 dt < \infty$, we get $|G[\gamma(t_j)] - G[\gamma(t_k')]| \to 0$, a contradiction. Thus, $\gamma(\cdot, t)$ converges uniquely to $\gamma_\infty$.\\
\textbf{Step 3.}
By Lemma 3.2, the length conservation holds: $L[\gamma_\infty] = L[\gamma_0]$. Lemma 3.1 gives $\int_\gamma k_\infty ds = \theta$, and from Lemma 4.2 we obtain the exponential estimate:
\begin{equation*}
\int_\gamma (k - \bar{k})^2 ds \leq Ce^{-\delta t}
\end{equation*}
These results imply $\gamma_\infty$ is either a straight line or a circular arc of a sector. Since a straight line violates the boundary condition (5), $\gamma_\infty$ must be the desired circular arc with central angle $\theta$.\\
\textbf{Step 4.}
Substitute the evolution equation of (5) into the expansion of $G[\gamma(t)]$, we have:
\begin{equation*}
\int_{-1}^1 |\gamma(u,\tau) - \gamma_\infty(u)|^2 du \leq C\int_\tau^T \|\partial_t \gamma\|_2 dt + G[\gamma(T)]
\end{equation*}
Letting $T \to \infty$ (noting $G[\gamma(T)] \to 0$ from Step 2) and applying the exponential estimate from Lemma 5.5, 5.6, we get:
\begin{equation*}
\int_{-1}^1 |\gamma(u,\tau) - \gamma_\infty(u)|^2 du \leq \frac{C}{\delta}e^{-\delta\tau}
\end{equation*}
Using a standard interpolation inequality, this $L^2$-convergence extends to the $C^\infty$-topology for all $l \in \mathbb{N}$, proving the exponential convergence of $\gamma(\cdot, t)$ to $\gamma_\infty$.
This completes the proof.
\end{pf}

\begin{remark}
The proof framework of Theorem 5.8 is largely consistent with that of \cite{1}. The key adjustments are: (1) the evolution equation of the curve diffusion flow (4) is substituted into the subconvergence analysis (Step 1) and convergence rate estimation (Step 4); (2) all relevant integral and differential estimates rely on Lemmas 5.5 and 5.6 established in the previous sections, rather than the estimates in \cite{1}. The core logic of convergence analysis (e.g., Arzel\`{a}-Ascoli theorem, contradiction argument for full convergence) remains unchanged from \cite{1}.
\end{remark}


\section*{Data availability Statement}
No datasets were generated or analyzed during the current study.

\section*{Acknowledgements}
The authors would like to thank the referees for careful reading of the manuscript and their valuable comments.
\section*{Conflict of interest} The author has no relevant financial or non-financial interests to disclose.









\end{document}